\documentclass[a4paper,12pt]{amsart}

\usepackage{amsfonts}
\usepackage{amsmath}
\usepackage{amssymb}

\usepackage{mathrsfs}
\usepackage{hyperref}

\numberwithin{equation}{section}
\theoremstyle{plain}
\newtheorem{thm}{Theorem}[section]

\theoremstyle{definition}
\newtheorem{defi}[thm]{Definition}

\newcommand{\be}{\begin{equation}}
\newcommand{\ee}{\end{equation}}

\def\R{{\mathbb R}}

\def\N{{\mathbb N_{0}^{2}}}

\def\C{{\mathbb C}}
\def\S3{{{\mathbb S}^3}}
\def\SU2{{{\rm SU}(2)}}
\def\Rn{{{\mathbb R}^n}}

\def\HS{{\mathtt{HS}}}

\def\p#1{{\left({#1}\right)}}

\def\jp#1{{\left\langle{#1}\right\rangle}}

\def\Dcal{{\mathcal D}}

\DeclareMathOperator{\Tr}{Tr}

\def\Rr{{\mathbb R}}

\def\C{{\mathbb C}}

\def\Rn{{\mathbb R}^n}
\def\R2n{{\mathbb R}^{2n}}

\def\S{{\mathcal S}}

\def\Rr{{\mathbb R}}
\def\Rn{{\mathbb R}^n}

\def\C{{\mathbb C}}
\def\A{{\mathbb A}}

\def\H{{\mathbb H}}

\def\R2{{\mathbb R}^2}
\def\R2n{{\mathbb R}^{2n}}

\def\S{{\mathcal S}}

\def\H{{\mathcal H}}

\def\T1{{\mathbb T}^1}

\newcommand{\eps}{\varepsilon}

\begin{document}

\title[Very weak solutions of wave equation for Landau Hamiltonian]
{Very weak solutions of wave equation for Landau Hamiltonian with irregular electromagnetic field}

\author[Michael Ruzhansky]{Michael Ruzhansky}
\address{
  Michael Ruzhansky:
  \endgraf
  Department of Mathematics
  \endgraf
  Imperial College London
  \endgraf
  180 Queen's Gate, London, SW7 2AZ
  \endgraf
  United Kingdom
  \endgraf
  {\it E-mail address} {\rm m.ruzhansky@imperial.ac.uk}
  }
\author[Niyaz Tokmagambetov]{Niyaz Tokmagambetov}
\address{
  Niyaz Tokmagambetov:
  \endgraf
    al--Farabi Kazakh National University
  \endgraf
  71 al--Farabi ave., Almaty, 050040
  \endgraf
  Kazakhstan,
  \endgraf
   and
  \endgraf
    Department of Mathematics
  \endgraf
  Imperial College London
  \endgraf
  180 Queen's Gate, London, SW7 2AZ
  \endgraf
  United Kingdom
  \endgraf
  {\it E-mail address} {\rm n.tokmagambetov@imperial.ac.uk}
 }

\thanks{The authors were supported in parts by the EPSRC
grant EP/K039407/1 and by the Leverhulme Grant RPG-2014-02,
as well as by the MESRK grant 0773/GF4. No new data was collected or generated during the course of research.}

\date{\today}


\begin{abstract}
In this paper we study the Cauchy problem for the Landau Hamiltonian wave equation, with time dependent irregular (distributional) electromagnetic field and similarly irregular velocity. For such equations, we describe the notion of a `very weak solution' adapted to the type of solutions that exist for
regular coefficients. The construction is based on considering Friedrichs--type mollifier of the coefficients and
corresponding classical solutions, and
their quantitative behaviour in the regularising parameter.
We show that even for distributional coefficients, the Cauchy problem does have a very weak solution,
and that this notion leads to classical or distributional  type solutions under conditions when
such solutions also exist.
\end{abstract}

\maketitle

\section{Introduction}

The purpose of this paper is to establish the well-posedness results for the wave equation for the Landau Hamiltonian with irregular electromagnetic field and similarly irregular velocity. We are especially interested in distributional irregularities appearing, for example, when modelling electric shocks by $\delta$-function type behaviour. While this leads to fundamental mathematical difficulties for the usual distributional interpretation of the equation (and of the Cauchy problem) due to impossibility of multiplication of distributions (see Schwartz \cite{S54}) we are able to establish the well-posedness using a notion of very weak solutions introduced in \cite{GR15b} in the context of space-invariant hyperbolic problems.
This notion also allows us to recapture the classical/distributional solution to the Cauchy problem for the Landau Hamiltonian under conditions when it does exist.

\smallskip
Thus, we consider a non-relativistic particle with mass $m$ and electric charge $e$ moving in a given electromagnetic field.
We concentrate on the 2D version and then indicate in Section \ref{SEC:n} the changes for the multidimensional case in $\mathbb R^{2d}$.
To describe the electromagnetic field in the plane one usually uses the electromagnetic scalar and vector potentials $q, \A$. In a brief description of the physical model below we follow \cite{ABGM15} to which we can also refer for some more details. 
The dynamics of a particle with mass $m$ and charge $e$ on the Euclidean $xy$--plane interacting with a perpendicular homogeneous electromagnetic field, is determined by the Hamiltonian (see \cite{LL77})
\begin{equation} \label{eq:Hamiltonian}
\H_{0}:=\frac{1}{2m} \p{i h \nabla-\frac{e}{c}\A}^{2}+eq,
\end{equation}
where $h$ denotes Planck's constant, $c$ is the speed of light and
$i$ the imaginary unit (see also Section \ref{SEC:n}). We denote by $2B>0$ the strength of the
magnetic field, and can choose the symmetric gauge given by
$$
\A=\p{- B y, B x}.
$$
As usual for mathematical models, we set $m=e=c=h=1$ in \eqref{eq:Hamiltonian}, which leads to the Landau Hamiltonian
$$
\H_{1}:=\H+q,
$$
where
\begin{equation} \label{eq:LandauHamiltonian}
\H:=\frac{1}{2} \p{\p{i\frac{\partial}{\partial x}-B
y}^{2}+\p{i\frac{\partial}{\partial y}+B x}^{2}},
\end{equation}
acting on the Hilbert space $L^{2}(\mathbb R^{2})$. It is well known
(see \cite{F28, L30}) that the spectrum of the
operator $\H$ consists of infinite number of eigenvalues with
infinite multiplicity of the form
\begin{equation} \label{eq:HamiltonianEigenvalues}
\lambda_{n}=\p{2n+1}B, \,\,\, n=0, 1, 2, \dots \,.
\end{equation}
These eigenvalues are called the Euclidean Landau levels. We denote the eigenspace of $\H$ corresponding to the eigenvalue $\lambda_{n}$ in \eqref{eq:HamiltonianEigenvalues} by
\begin{equation} \label{eq:HamiltonianEigenspaces}
\mathcal A_{n}(\mathbb R^{2})=\{\varphi\in L^{2}(\mathbb R^{2}), \,\, \H \varphi=\lambda_{n}\varphi\}.
\end{equation}
The following functions form an orthogonal basis for $\mathcal
A_{n}(\mathbb R^{2})$ (see \cite{ABGM15, HH13}):
{\small
\begin{equation}
\label{eq:HamiltonianBasis} \left\{
\begin{split}
e^{1}_{k, n}(x,y)&=\sqrt{\frac{n!}{(n-k)!}}B^{\frac{k+1}{2}}\exp\Big(-\frac{B(x^{2}+y^{2})}{2}\Big)(x+iy)^{k}L_{n}^{(k)}(B(x^{2}+y^{2})), \,\,\, 0\leq k, {}\\
e^{2}_{j, n}(x,y)&=\sqrt{\frac{j!}{(j+n)!}}B^{\frac{n-1}{2}}\exp\Big(-\frac{B(x^{2}+y^{2})}{2}\Big)(x-iy)^{n}L_{j}^{(n)}(B(x^{2}+y^{2})), \,\,\, 0\leq j,
\end{split}
\right.
\end{equation}}
where $L^{(\alpha)}_{n}$ is the Laguerre polynomial defined as
$$
L^{(\alpha)}_{n}(t)=\sum_{k=0}^{n}(-1)^{k}C_{n+\alpha}^{n-k}\frac{t^{k}}{k!}, \,\,\, \alpha>-1.
$$
To simplify the notation further we denote
\begin{equation}\label{EQ:eks}
e^{k}_{\xi}:=e^{k}_{j, n} \,\,\ \hbox{for} \,\,\,  \xi=(j,n), \,\,\, j, n=0, 1, 2, ...; \,\,\, k=1, 2.
\end{equation}

To finish the brief literature review, we refer to \cite[Remark 1]{ABGM15}
for an explanation of relations of the basis
\eqref{eq:HamiltonianBasis}
to Feynman and
Schwinger's work on
finding matrix elements of the displacement
operator (see also Perelomov \cite[p. 35]{P86}),
the relations to complex Hermite polynomials (\cite{I16}),
to quantization questions \cite{ABG12,BG14,CGG10}, 
to time-frequency analysis, partial differential equations and planar
point processes, see \cite{A10,G08,HH13}, respectively.

We can also mention papers \cite{K16, N96}, where the authors investigated
properties of eigenfunctions of perturbed Hamiltonians, and in
\cite{S14, KP04, M91, PR07, PRV13, LR14, RT08}
asymptotics of the eigenvalues for perturbed Landau Hamiltonians
were described.

In this paper we are interested in the wave equation for the Landau Hamiltonian with time-dependent irregular electric potential and varying in time electromagnetic field.
More precisely, for a distributional propagation speed function $a=a(t)\geq 0$ and for the distributional electromagnetic scalar potential $q=q(t)$, we consider the
Cauchy problem for the Landau Hamiltonian $\H$ in the form
\begin{equation}\label{CPa}
\left\{ \begin{split}
\partial_{t}^{2}u(t,x)+a(t)[\H+q(t)] u(t,x)&=0, \; (t,x)\in [0,T]\times \mathbb R^{2},\\
u(0,x)&=u_{0}(x), \; x\in \mathbb R^{2}, \\
\partial_{t}u(0,x)&=u_{1}(x), \; x\in \mathbb R^{2}.
\end{split}
\right.
\end{equation}
A special feature of our analysis is that we want to allow $a$ and $q$ to be distributions. For instance, if the electric potential produces shocks these can be modelled with $\delta$-distributions, for example by taking $q=\delta_{1}$, the $\delta$-distribution at time $t=1$.
Moreover, if the velocity $a(t)$ also contained $\delta$-type terms, as an example of such an equation we could consider
\begin{equation}\label{CPa-ex1}
\left\{ \begin{split}
\partial_{t}^{2}u(t,x)+\delta_{1}[\H+\delta_{1}] u(t,x)&=0, \; (t,x)\in [0,T]\times \mathbb R^{2},\\
u(0,x)&=u_{0}(x), \; x\in \mathbb R^{2}, \\
\partial_{t}u(0,x)&=u_{1}(x), \; x\in \mathbb R^{2}.
\end{split}
\right.
\end{equation}
Moreover, we could also look at discontinuous speeds given by e.g. the Heaviside function $h(t)$ such that $h(t)=1$ for $t<1$ and $h(t)=2$ for $t\geq 1$, and singular electric fields, e.g.
$q(t)=\delta_{1}+h(t)$, in which case the Cauchy problem \eqref{CPa} would take the form
\begin{equation}\label{CPa-ex2}
\left\{ \begin{split}
\partial_{t}^{2}u(t,x)+h(t)[\H+\delta_{1}+h(t)] u(t,x)&=0, \; (t,x)\in [0,T]\times \mathbb R^{2},\\
u(0,x)&=u_{0}(x), \; x\in \mathbb R^{2}, \\
\partial_{t}u(0,x)&=u_{1}(x), \; x\in \mathbb R^{2}.
\end{split}
\right.
\end{equation}
The physical problem that we are interested in is as follows:
\begin{center}
{\em How to understand the Cauchy problems \eqref{CPa}-\eqref{CPa-ex2} and their well-posedness?}
\end{center}
There are several difficulties already at the fundamental level for such problems, first of all in general impossibility of multiplying distributions due to the famous Schwartz' impossibility result
\cite{S54}. Second, even if we could somehow make sense of the product $aq$ being a distribution by e.g. imposing wave front conditions, we would still have to multiply it with $u(t,x)$ which, a-priori, may also have singularities in $t$, thus leading to another multiplication problem. Moreover, another difficulty (for the global in space analysis of \eqref{CPa}) is that the coefficients of $\H$ increase in space thus leading to potential problems at infinity if we treat the problem only locally.

In our analysis we assume that $a$ is a positive distribution so that the Cauchy problem \eqref{CPa} is of hyperbolic type, at least when $a$ and $q$ are regular. More precisely, we will assume that there exists a constant $a_{0}>0$ such that
$$
a\geq a_{0}>0,
$$
where $a\geq a_{0}$ means that $a-a_{0}\geq0$, i.e. $\langle a-a_{0}, \psi\rangle\geq0$ for all $\psi\in C^\infty_0(\mathbb R)$, $\psi\ge 0$. Incidentally, the structure theory of distributions implies that $a$ is a Radon measure but this does not remove the multiplication problems or problem with understanding the meaning of the well-posedness of the Cauchy problem \eqref{CPa}.

Nevertheless, we are able to study the well-posedness of \eqref{CPa} using an adaptation of the notion of very weak solutions introduced in \cite{GR15b} in the context of hyperbolic problems with distributional coefficients in $\mathbb R^{n}$.

As noted, the equation \eqref{CPa} can not be, in general, understood distributionally, so we are forced to weaken the notion of solutions. However, we want to do it in a way so that we can recapture classical solutions should they exist.
Thus, in this paper we will show the following facts:

\begin{itemize}
\item The Cauchy problem \eqref{CPa} admits a {\em very weak solution} even for distributional-type Cauchy data $u_{0}$ and $u_{1}$. The very weak solution is unique in an appropriate sense.
\item If the coefficients $a$ and $q$ are regular so that the Cauchy problem \eqref{CPa} has a `classical' solution, the very weak solution recaptures this classical solution in the limit of the regularising parameter. This shows that the introduced notion of a very weak solution is consistent with classical solutions should the latter exist.
\item When the classical solution does not exist, the very weak solution comes with an explicit numerical scheme modelling the limiting behaviour of regularised solutions.
\end{itemize}

We also note that at the same time our analysis will yield results for the modified problem
\begin{equation}\label{CPb}
\left\{ \begin{split}
\partial_{t}^{2}u(t,x)+a(t)\H u(t,x)+q(t) u(t,x)&=0, \; (t,x)\in [0,T]\times \mathbb R^{2},\\
u(0,x)&=u_{0}(x), \; x\in \mathbb R^{2}, \\
\partial_{t}u(0,x)&=u_{1}(x), \; x\in \mathbb R^{2},
\end{split}
\right.
\end{equation}
for distributions $a,q$ with $a\geq a_{0}>0$ for some constant $a_{0}.$

For second order operators $\H$ independent of $x$ the Cauchy problems of this type have been intensively studied, however for more regular (starting from H\"older) coefficients, see for example
\cite{Cicognani-Colombini:JDE-2013,Colombini-deGiordi-Spagnolo-Pisa-1979,
Colombini-del-Santo-Kinoshita:ASNS-2002,
Colombini-del-Santo-Reissig:BSM-2003,DS} and references therein.
For the setting of distributional coefficients see \cite{GR15b}.

The analysis of this paper is different from the one in \cite{GR15b} that was adapted to constant coefficients in $\Rn$. At the same time, the techniques of the present paper may be extended to treat more general operators, however, since such analysis is more abstract and requires more background material, it will appear elsewhere.

The description of appearing function spaces is carried out in the spirit of \cite{DR16} using the general development of nonharmonic type analysis carried out by the authors in \cite{RT16} which is, however, `harmonic' in the present setting. The treatment of the global well-posedness in the appearing function spaces is an extension of the method developed in \cite{GR15} in the context of compact Lie groups.

In Section \ref{SEC:Inhom} we show an extension of the construction to also consider the inhomogeneous wave equation
\begin{equation}\label{CPa-InH0}
\left\{ \begin{split}
\partial_{t}^{2}u(t,x)+a(t)[\H+q(t)] u(t,x)&=f(t,x), \; (t,x)\in [0,T]\times \mathbb R^{2},\\
u(0,x)&=u_{0}(x), \; x\in \mathbb R^{2}, \\
\partial_{t}u(0,x)&=u_{1}(x), \; x\in \mathbb R^{2},
\end{split}
\right.
\end{equation}

The structure of the paper is as follows. In Section \ref{SEC:results} we formulate our main results.
In Section \ref{SEC:Prelim} we discuss elements of the global Fourier analysis associated to the Landau Hamiltonian as a special case of abstract constructions
that have been developed in \cite{RT16}.
In Section \ref{SEC:reduction} we prove Theorem \ref{theo_case_1} and
in Section \ref{SEC:pf2} we prove
Theorem \ref{theo_vws}. In Section \ref{SEC:cons} we establish the uniqueness of very weak solutions and their consistence with `classical' solutions when they exist.
In Section \ref{SEC:Inhom} we give an extension of our constructions to the inhomogeneous wave equation. In Section \ref{SEC:n} we discuss an extension to higher dimensions, namely, to the Landau Hamiltonian in $\mathbb R^{2d}$.

\section{The main results}
\label{SEC:results}

In our results below, concerning the Cauchy problem \eqref{CPa}, as the preliminary step we first
carry out the analysis in the strictly hyperbolic regular case $a(t)\ge a_0>0$, for differentiable $a, q$ with $\partial_{t}a,\partial_{t}q\in L^{\infty}([0,T])$.
Thus, we denote by $L_{1}^{\infty}([0,T])$ the space functions $a\in L^{\infty}([0,T])$ with $\partial_{t}a\in L^{\infty}([0,T])$.
In this case we obtain the well-posedness in the Sobolev spaces  ${H}^{s}_\H$ associated to the operator $\H$: we define the Sobolev spaces $H^s_\H$ associated to
$\H$, for any $s\in\Rr$, as the space
$$
H^s_\H:=\left\{ f\in\Dcal'_{\H}(\mathbb R^{2}): \H^{s/2}f\in
L^2(\mathbb R^{2})\right\},
$$
with the norm $\|f\|_{H^s_\H}:=\|\H^{s/2}f\|_{L^2}.$
The global space of distributions $\Dcal'_{\H}(\mathbb R^{2})$ is defined in
Section \ref{SEC:Prelim}.

\begin{thm}
\label{theo_case_1}
Assume that $a, q\in L_{1}^{\infty}([0,T])$ are such that $a(t)\ge a_0>0$ and $q(t)\geq 0$.
For any $s\in\Rr$, if the Cauchy data satisfy
$(u_0,u_1)\in {H}^{1+s}_\H \times {H}^{s}_\H$,
then the Cauchy problem \eqref{CPa} has a unique solution
$u\in C([0,T],{H}^{1+s}_\H) \cap
C^1([0,T],{H}^{s}_\H)$ which satisfies the estimate
\begin{equation}
\label{case_1_last-est}
\|u(t,\cdot)\|_{{H}^{1+s}_\H}^2+\| \partial_t u(t,\cdot)\|_{{H}^s_\H}^2\leq
C (\| u_0\|_{{H}^{1+s}_\H}^2+\|u_1\|_{{H}^{s}_\H}^2).
\end{equation}
The same result is true also for the Cauchy problem \eqref{CPb}.
\end{thm}

\vspace{3mm}

Anticipating the material of the next section, using
Plancherel's identity \eqref{EQ:Plancherel}, in our case we can express the Sobolev norm as
\begin{equation}\label{EQ:Sob1}
\|f\|_{H^s_\H}= \p{ \sum_{\xi\in\N} (B+2B\xi_{2})^{s}
\sum_{j=1}^{2} \left| \int_{\mathbb
R^{2}}f(x)\overline{e^{j}_{\xi}(x)}dx\right|^2  }^{1/2},
\end{equation}
with $e_{\xi}^{j}$ as in \eqref{EQ:eks}.

In Theorem \ref{theo_case_1} the assumption of $q$ being real-valued is actually enough to assure the well-posedness, however, we assume  that $q\geq 0$ to facilitate the proofs of the distributional results later.

We now describe the notion of very weak solutions and formulate the corresponding results for distributions $a,q\in\Dcal'([0,T])$.
The first main idea is to start from the distributional coefficient $a$ to regularise it by convolution with a suitable mollifier $\psi$ obtaining families of smooth functions $(a_{\eps})_\eps$, namely
\begin{equation}\label{EQ:regs}
a_{\eps}=a\ast\psi_{\omega(\eps)},
\end{equation}
where  $\psi_{\omega(\eps)}(t)=\omega(\eps)^{-1}\psi(t/\omega(\eps))$ and $\omega(\eps)$ is a positive function converging to $0$ as $\eps\to 0$ to be chosen later.
Here $\psi$ is a Friedrichs--mollifier, i.e.  $\psi\in C^\infty_0(\mathbb R)$, $\psi\ge 0$ and $\int\psi=1$.
It turns out that the net $(a_{\eps})_\eps$ is $C^\infty$-\emph{moderate}, in the sense that its $C^\infty$-seminorms can be estimated by a negative power of $\eps$. More precisely, we will make use of the following notions of moderateness.

In the sequel, the notation $K\Subset\mathbb R$ means that $K$ is
a compact set in $\mathbb R$.

\begin{defi}
\label{def_mod_intro}
\leavevmode
\begin{itemize}
\item[(i)]
A net of functions $(f_\eps)_\eps\in C^\infty(\mathbb R)^{(0,1]}$
(i.e. $(f_\eps)_{\eps\in (0,1]}\subset C^\infty(\mathbb R)$)
is said to be $C^\infty$-moderate if for all $K\Subset\mathbb R$ and for all $\alpha\in\mathbb N_{0}$ there exist $N=N_{\alpha}\in\mathbb N_{0}$ and $c=c_{\alpha}>0$ such that
\[
\sup_{t\in K}|\partial^\alpha f_\eps(t)|\le c\eps^{-N-\alpha},
\]
for all $\eps\in(0,1]$.
\item[(ii)] A net of functions $(u_\eps)_\eps\in C^\infty([0,T];{H}^{s}_\H)^{(0,1]}$ is said to be $C^\infty([0,T];{H}^{s}_\H)$-moderate if there exist $N\in\mathbb N_{0}$ and $c_k>0$ for all $k\in\mathbb N_{0}$ such that
\[
\|\partial_t^k u_\eps(t,\cdot)\|_{H^{s}_\H}\le c_k\eps^{-N-k},
\]
for all $t\in[0,T]$ and $\eps\in(0,1]$.
\end{itemize}
\end{defi}
We note that the conditions of moderateness are natural in the sense that regularisations of distributions are moderate, namely we can regard
\begin{equation}\label{EQ:incls}
\textrm{ compactly supported distributions } \mathcal{E}'(\mathbb R)\subset \{C^\infty \textrm{-moderate families}\}
\end{equation}
by the structure theorems for distributions.

Thus, while a solution to the Cauchy problems may not exist in the space of distributions on the left hand side of \eqref{EQ:incls}, it may still exist (in a certain appropriate sense)
in the space on its right hand side. The moderateness assumption will be crucial allowing to recapture the solution as in \eqref{theo_case_1} should it exist. However, we note that regularisation with standard Friedrichs mollifiers will not be
sufficient, hence the introduction of a family $\omega(\eps)$ in the above regularisations.

We can now introduce a notion of a `very weak solution' for the Cauchy problem \eqref{CPa}.
\begin{defi}
\label{def_vws}
Let $s\in\mathbb R$ and $u_{0},u_{1}\in {H}^{s}_\H$. The net $(u_\eps)_\eps\in C^\infty([0,T];{H}^{s}_\H)$ is {\em a very weak solution of order $s$} of the
Cauchy problem \eqref{CPa} if there exist
\begin{itemize}
\item[]
$C^\infty$-moderate regularisations $a_{\eps}$ and $q_{\eps}$ of the coefficients $a$ and $q$,
\end{itemize}
such that $(u_\eps)_\eps$ solves the regularised problem
\begin{equation*}\label{CPbb}
\left\{ \begin{split}
\partial_{t}^{2}u_{\eps}(t,x)+a_{\eps}(t)[\H+q_{\eps}(t)] u_{\eps}(t,x)&=0, \; (t,x)\in [0,T]\times \mathbb R^{2},\\
u_{\eps}(0,x)&=u_{0}(x), \; x\in \mathbb R^{2}, \\
\partial_{t}u_{\eps}(0,x)&=u_{1}(x), \; x\in \mathbb R^{2},
\end{split}
\right.
\end{equation*}
for all $\eps\in(0,1]$, and is $C^\infty([0,T];{H}^{s}_\H)$-moderate.
\end{defi}
We note that according to Theorem \ref{theo_case_1} the regularised Cauchy problem
\eqref{CPbb} has a unique solution satisfying estimate \eqref{case_1_last-est}.

In \cite{GR15b} the authors studied weakly hyperbolic second order equations with time-dependent irregular coefficients, assuming that the coefficients are distributions. For such equations, the authors of  \cite{GR15b} introduced the notion of a `very weak solution' adapted to
the type of solutions that exist for regular coefficients. We now apply a modification of this notion to the Cauchy problems \eqref{CPa} and \eqref{CPb}.

\vspace{2mm}

In the following theorem we assume that $a$ is a strictly positive distribution, which means that there exists a constant $a_{0}>0$ such that $a-a_{0}$ is a positive distribution. In other words,
$$
a\geq a_{0}>0,
$$
where $a\geq a_{0}$ means that $a-a_{0}\geq0$, i.e. $\langle a-a_{0}, \psi\rangle\geq0$ for all $\psi\in C^\infty_0(\mathbb R)$, $\psi\ge 0$. Analogously, a distribution $q$ is called real-valued if
$\langle q, \psi\rangle\in\mathbb R$ for all real-valued test functions $\psi\in C^\infty_0(\mathbb R)$, and $q$ is positive if $\langle q, \psi\rangle\geq 0$ whenever $\psi\geq 0$.

The main results of this paper can be summarised as the following solvability statement complemented by the uniqueness and consistency in Theorems \ref{theo_consistency-1} and \ref{theo_consistency-2}.

\begin{thm}[Existence]
\label{theo_vws}
Let the coefficients $a$ and $q$ of the Cauchy problem \eqref{CPa} be positive distributions
with compact support included in $[0,T]$, such that $a\ge a_{0}$ for some constant $a_{0}>0$.
Let $s\in\mathbb R$ and let the Cauchy data $u_0, u_1$ be in ${H}^{s+1}_\H$.
Then the Cauchy problem \eqref{CPa} has a very weak solution of order $s$.

The same result is true also for the Cauchy problem \eqref{CPb}.
\end{thm}

Since $s$ is allowed to be negative, the Cauchy data are allowed to be $\H$-distributions (i.e. elements of ${H}^{s}_\H$ with negative $s$).
In Theorem \ref{theo_consistency-1} we show that the very weak solution is unique in an appropriate sense.

But now let us formulate the theorem saying that very weak solutions recapture the classical solutions in the case the latter exist. This happens, for example, under conditions of Theorem \ref{theo_case_1}. So, we can compare the solution given by Theorem \ref{theo_case_1} with the very weak solution in Theorem \ref{theo_vws} under assumptions when Theorem \ref{theo_case_1} holds.

\begin{thm}[Consistency]
\label{theo_consistency-2}
Assume that $a, q\in L_{1}^{\infty}([0,T])$ are such that $a(t)\ge a_0>0$ and $q(t)\geq 0$.
Let $s\in\mathbb R$, and consider the Cauchy problem
\begin{equation}\label{Consistency:EQ:1-2}
\left\{ \begin{split}
\partial_{t}^{2}u(t,x)+a(t)[\H+q(t)] u(t,x)&=0, \; (t,x)\in [0,T]\times \mathbb R^{2},\\
u(0,x)&=u_{0}(x), \; x\in \mathbb R^{2}, \\
\partial_{t}u(0,x)&=u_{1}(x), \; x\in \mathbb R^{2},
\end{split}
\right.
\end{equation}
with $(u_0,u_1)\in {H}^{1+s}_\H \times {H}^{s}_\H$. Let $u$ be a very weak solution of
\eqref{Consistency:EQ:1-2}. Then for any regularising families $a_{\eps}, q_{\eps}$ in Definition \ref{def_vws}, any representative $(u_\eps)_\eps$ of $u$ converges in $C([0,T],{H}^{1+s}_\H) \cap
C^1([0,T],{H}^{s}_\H)$ as $\eps\rightarrow0$
to the unique classical solution in $C([0,T],{H}^{1+s}_\H) \cap
C^1([0,T],{H}^{s}_\H)$ of the Cauchy problem \eqref{Consistency:EQ:1-2}
given by Theorem \ref{theo_case_1}.

The same statement holds for \eqref{Consistency:EQ:1-2} replaced by
\eqref{CPb}.
\end{thm}

Here the very weak solution is understood according to Definition \ref{def_vws}.
We now proceed with preparation for proving theorems in this section.

\section{Fourier analysis for the Landau Hamiltonian}
\label{SEC:Prelim}

In this section we recall the necessary elements of the global Fourier analysis
that has been developed in \cite{RT16} applied to the present setting.
Although the domain $\mathbb R^{2}$
in our setting is unbounded, the following constructions carry over without any significant changes.
Moreover, there is a significant simplification since the appearing Fourier analysis is
self-adjoint.
A more general version of these constructions under weaker conditions can be found in \cite{RT16a}.
For application of the general non-self-adjoint analysis to the spectral analysis we refer to \cite{DRT16}.

The space $C_{\H}^{\infty}(\mathbb R^{2}):={\rm Dom}({\H}^{\infty})$ is called the space of test functions for ${\H}$, where we define
$$
{\rm Dom}({\H}^{\infty}):=\bigcap_{k=1}^{\infty}{\rm Dom}({\H}^{k}),
$$
where ${\rm Dom}({\H}^{k})$ is the domain of the operator ${\H}^{k}$, in turn defined as
$$
{\rm Dom}({\H}^{k}):=\{f\in L^{2}(\mathbb R^{2}): \,\,\, {\H}^{j}f\in {\rm Dom}({\H}), \,\,\, j=0, \,1, \, 2, \ldots,
k-1\}.
$$
The Fr\'echet topology of $C_{{\H}}^{\infty}(\mathbb R^{2})$ is given by the family of norms
\begin{equation}\label{EQ:L-top}
\|\varphi\|_{C^{k}_{{\H}}}:=\max_{j\leq k}
\|{\H}^{j}\varphi\|_{L^2(\mathbb R^{2})}, \quad k\in\mathbb N_0,
\; \varphi\in C_{{\H}}^{\infty}(\mathbb R^{2}).
\end{equation}
The space of
${\H}$-distributions
$$\mathcal D'_{\H}(\mathbb R^{2}):=\mathcal L(C_{\H}^{\infty}(\mathbb R^{2}),
\mathbb C)$$ is the space of all linear continuous functionals on
$C_{\H}^{\infty}(\mathbb R^{2})$. For
$w\in\mathcal D'_{\H}(\mathbb R^{2})$ and $\varphi\in C_{\H}^{\infty}(\mathbb R^{2})$,
we shall write
$$
w(\varphi)=\langle w, \varphi\rangle.
$$
For any $\psi\in C_{{\H}}^{\infty}(\mathbb R^{2})$, the functional
$$
C_{\H}^{\infty}(\mathbb R^{2})\ni \varphi\mapsto\int_{\mathbb R^{2}}{\psi(x)} \, \varphi(x)\, dx
$$
is an ${\H}$-distribution, which gives an embedding $\psi\in
C_{{\H}}^{\infty}(\mathbb R^{2})\hookrightarrow\mathcal D'_{\H}(\mathbb R^{2})$.

Taking into account the fact that the eigenfunctions of the Landau Hamiltonian in \eqref{eq:HamiltonianBasis} come in pairs, it will be convenient to group them together in the way
suggested by the notation \eqref{EQ:eks}. This leads to the following definitions.
Let $\mathcal S(\N)$ denote the space of rapidly decaying
functions $\varphi:\N\rightarrow\mathbb C^{2\times2}$ of the form
$$
\varphi:=
\p{\begin{matrix} \varphi_{11}  & 0 {}\\
0 &\varphi_{22}
\end{matrix}}.
$$

That is,
$\varphi\in\mathcal S(\N)$ if for any $M<\infty$ there
exists a constant $C_{\varphi, M}$ such that
$$
|\varphi(\xi)|\leq C_{\varphi, M}\langle\xi\rangle^{-M}
$$
holds for all $\xi\in\N$, where we denote
$$\langle\xi\rangle:=|\sqrt{\lambda_{\xi_{2}}}|=|\sqrt{(2\xi_{2}+1)B}|.$$

The topology on $\mathcal
S(\N)$ is given by the seminorms $p_{k}$, where
$k\in\mathbb N_{0}$ and
$$
p_{k}(\varphi):=\sup_{\xi\in\N}\langle\xi\rangle^{k}|\varphi(\xi)|.
$$

We now define the $\H$-Fourier transform on $C_{\H}^{\infty}(\mathbb R^{2})$ as the mapping
$$
(\mathcal F_{\H}f)(\xi)=(f\mapsto\widehat{f}):
C_{\H}^{\infty}(\mathbb R^{2})\rightarrow\mathcal S(\N)
$$
by the formula
\begin{equation}
\label{FourierTr}
\widehat{f}(\xi):=(\mathcal F_{\H}f)(\xi)=\int_{\mathbb R^{2}}f(x)\overline{e_{\xi}(x)}dx,
\end{equation}
where
$$
e_{\xi}(x)=\p{\begin{matrix} e^{1}_{\xi}(x)  & 0 {}\\
0 & e^{2}_{\xi}(x)
\end{matrix}}.
$$

The $\H$-Fourier transform
$\mathcal F_{\H}$ is a bijective homeomorphism from $C_{{\H}}^{\infty}(\mathbb R^{2})$ to $\mathcal S(\N)$.
Its inverse  $$\mathcal F_{\H}^{-1}: \mathcal S(\N)
\rightarrow C_{\H}^{\infty}(\mathbb R^{2})$$ is given by
\begin{equation}
\label{InvFourierTr} (\mathcal F^{-1}_{{\H}}h)(x)=\sum_{\xi \in \N}\Tr\p{h(\xi)e_{\xi}(x)},\quad h\in\mathcal S(\N),
\end{equation}
so that the Fourier inversion formula becomes
\begin{equation}
\label{InvFourierTr0}
f(x)=\sum_{\xi\in\N}\Tr\p{\widehat{f}(\xi)e_{\xi}(x)}
\quad \textrm{ for all } f\in C_{{\H}}^{\infty}(\mathbb R^{2}).
\end{equation}

The Plancherel identity takes the form
\begin{equation}\label{EQ:Plancherel}
\|f\|_{L^{2}(\mathbb R^{2})}=\p{\sum_{\xi\in\N}
\|\widehat{f}(\xi)\|_{\HS}^{2}}^{1/2}=:
\|\widehat{f}\|_{\ell^{2}(\N)},
\end{equation}
which we can take as the definition of the norm on the Hilbert space
$\ell^{2}(\N)$, and where
$ \|\widehat{f}(\xi)\|_{\HS}^{2}=\Tr(\widehat{f}(\xi)\overline{\widehat{f}(\xi))}$ is the
Hilbert--Schmidt norm of the matrix $\widehat{f}(\xi)$.

One can readily check that test functions and distributions on $\mathbb R^{2}$ can be characterised in terms of
their Fourier coefficients. Thus, we have
$$
f\in C^{\infty}_{\H}(\mathbb R^{2})\Longleftrightarrow
\forall N \;\exists C_{N} \textrm{ such that }
\|\widehat{f}(\xi)\|_{\HS}\leq C_{N} \jp{\xi}^{-N}
\textrm{ for all } \xi\in\N.
$$
Also, for distributions, we have
$$
u\in \Dcal'_{\H}(\mathbb R^{2})
\Longleftrightarrow
\exists M \;\exists C \textrm{ such that }
\|\widehat{u}(\xi)\|_{\HS} \leq C\jp{\xi}^{M}
\textrm{ for all } \xi\in\N .
$$

In general, given a linear continuous operator $L:C^{\infty}_{\H}(\mathbb R^{2})\to C^{\infty}_{\H}(\mathbb R^{2})$
(or even $L:C^{\infty}_{\H}(\mathbb R^{2})\to \Dcal'_{\H}(\mathbb R^{2})$), we can define its matrix symbol by
$\sigma_{L}(x,\xi):=e_{\xi}(x)^{-1} (L e_\xi)(x)\in \C^{2\times2}$, where
$L e_\xi$ means that we apply $L$ to the matrix components of $e_\xi(x)$, provided that
$e_{\xi}(x)$ is invertible in a suitable sense.
In this case we may prove that
\begin{equation}\label{EQ:T-op}
Lf(x)=\sum_{\xi\in\N} \Tr\p{e_\xi(x)\sigma_{L}(x,\xi)\widehat{f}(\xi)}.
\end{equation}
The correspondence between operators and symbols is one-to-one.
The quantization \eqref{EQ:T-op} has been extensively studied in
\cite{Ruzhansky-Turunen:BOOK,Ruzhansky-Turunen:IMRN} in the setting of compact Lie groups, and in \cite{RT16} in the setting of (non-self-adjoint) boundary value problems, to which we
may refer for its properties and for the corresponding symbolic calculus.

However, the situation with the Landau Hamiltonian is now much simpler since this operator can be treated as an `invariant' operator in the corresponding global calculus.
The operator $\H$ acts as a Fourier multiplier in its own Fourier calculus, therefore its symbol
$\sigma_\H(\xi)$ is independent of $x$, and since $\H$ is formally self-adjoint and positive
we can always write it in the form
\begin{equation}\label{EQ:subL-symbol}
\sigma_{\H}(\xi)=
\left(\begin{matrix}
\nu_1^2(\xi) &  0\\
0  & \nu_2^2(\xi)
\end{matrix}\right),
\end{equation}
for some $\nu_j(\xi)\geq 0$. Indeed, we have
$\nu_{j}^{2}(\xi)=B(1+2\xi_{2})$ for $j=1,2$.

Consequently, we can also define Sobolev spaces $H^s_\H$ associated to
$\H$. Thus, for any $s\in\Rr$, we set
\begin{equation}\label{EQ:HsL}
H^s_\H:=\left\{ f\in\Dcal'_{\H}(\mathbb R^{2}): \H^{s/2}f\in
L^2(\mathbb R^{2})\right\},
\end{equation}
with the norm $\|f\|_{H^s_\H}:=\|\H^{s/2}f\|_{L^2}.$ Using
Plancherel's identity \eqref{EQ:Plancherel}, we can write
\begin{multline}\label{EQ:Hsub-norm}
\|f\|_{H^s_\H}=\|\H^{s/2}f\|_{L^2}=\p{\sum_{\xi\in\N} \|\sigma_\H(\xi)^{s/2}\widehat{f}(\xi)\|_\HS^2}^{1/2} \\
\,\,\,\,\,\,\,\,\,\,\,\,\,\,\,\,\,\,\,\,\,\,\,\,\,\,\,\,\,\,\,\,\,\,
= \p{ \sum_{\xi\in\N}  (B+2B\xi_{2})^{s} \sum_{j=1}^{2}
|\widehat{f}(\xi)_{jj}|^2  }^{1/2}\\
= \p{ \sum_{\xi\in\N}
(B+2B\xi_{2})^{s} \sum_{j=1}^{2} \left| \int_{\mathbb
R^{2}}f(x)\overline{e^{j}_{\xi}(x)}dx\right|^2 }^{1/2},
\end{multline}
justifying the expression \eqref{EQ:Sob1}.

\section{Proof of Theorem \ref{theo_case_1}}
\label{SEC:reduction}

We will prove the result for the Cauchy problem \eqref{CPa} since equation \eqref{CPb} can be treated by the same argument with minor modification.

The operator $\H$  has the symbol \eqref{EQ:subL-symbol}, which we can write
in matrix components as
$$
\sigma_{\H}(\xi)_{mk}=(B+2B\xi_{2})\delta_{mk}, \quad 1\leq m,k\leq
2,
$$
with $\delta_{mk}$ standing for the Kronecker's delta.
Taking the $\H$-Fourier transform of
\eqref{CPa}, we obtain the collection of Cauchy problems for
matrix-valued Fourier coefficients:
\begin{equation}\label{CPa-FC}
\partial_{t}^{2}\widehat{u}(t,\xi)+a(t)[\sigma_{\H}(\xi)+q(t)\textrm{I}]\widehat{u}(t,\xi)=0,
\quad \xi\in\N,
\end{equation}
where $\textrm{I}$ is the identity $2\times2$ matrix.
Writing this in the matrix form, we see that this is equivalent to the system
$$
\partial_{t}^{2} \widehat{u}(t,\xi)+
a(t) \left(\begin{matrix}
(q(t)+B+2B\xi_{2}) &  0 \\
0  &  (q(t)+B+2B\xi_{2})
\end{matrix}\right) \widehat{u}(t,\xi)=0.
$$
Rewriting \eqref{CPa-FC} in terms of matrix coefficients
$\widehat{u}(t,\xi)=\left(\widehat{u}(t,\xi)_{mk}\right)_{1\leq m,k\leq 2}$,
we get the equations
\begin{equation}\label{EQ:WE-scalars}
\partial_{t}^{2} \widehat{u}(t,\xi)_{mk}+ a(t) (q(t)+B+2B\xi_{2})
 \widehat{u}(t,\xi)_{mk}=0,\qquad \xi\in\N,\;
1\leq m,k\leq 2.
\end{equation}
The main point of our further analysis is that we can make an individual
treatment of the equations in \eqref{EQ:WE-scalars} and then collect the estimates together using the $\H$-Plancherel theorem.

Thus, let us fix $\xi\in\N$ and $m,k$ with $1\leq m,k\leq 2$,
and let us denote
$$\widehat{v}(t,\xi):=\widehat{u}(t,\xi)_{mk}.$$
We then study the Cauchy problem
\begin{equation}\label{EQ:WE-v}
\partial_{t}^{2} \widehat{v}(t,\xi)+ a(t) (q(t)+B+2B\xi_{2})
 \widehat{v}(t,\xi)=0,\,\,\,
 \widehat{v}(t,\xi)=\widehat{v}_{0}(\xi), \;
 \partial_{t}\widehat{v}(t,\xi)=\widehat{v}_{1}(\xi),
\end{equation}
with $\xi,m$ being parameters, and want to derive estimates
for $\widehat{v}(t,\xi)$. Combined with the characterisation \eqref{EQ:Hsub-norm} of
Sobolev spaces this will yield the well--posedness
results for the original Cauchy problem \eqref{CPa}.

In the sequel, for fixed $m$, we set
\begin{equation}\label{xi_l}
\nu^{2}(\xi):=(B+2B\xi_{2}).
\end{equation}
Hence, the equation in \eqref{EQ:WE-v} can be written as
\begin{equation}
\label{eq_xi}
\partial_{t}^{2} \widehat{v}(t,\xi)+a(t)\nu^{2}(\xi)\left[1+\frac{q(t)}{\nu^{2}(\xi)}\right]\widehat{v}(t,\xi)=0.
\end{equation}
We now proceed
with a standard reduction to a first order system of this equation and define the corresponding energy.
The energy estimates will be given in terms of $t$ and $\nu(\xi)$ and we then go back to $t$,
$\xi$ and $m$ by using \eqref{xi_l}.

We can now do the natural energy construction for \eqref{eq_xi}. We use the transformation
\[
\begin{split}
V_1&:=i\nu(\xi)\widehat{v},\\
V_2&:= \partial_t \widehat{v}.
\end{split}
\]
It follows that the equation \eqref{eq_xi} can be written as the first order system
\begin{equation}\label{EQ:system}
\partial_t V(t,\xi)=i\nu(\xi) A(t,\xi)V(t,\xi),
\end{equation}
where $V$ is the column vector with entries $V_1$ and $V_2$ and
$$
A(t,\xi)=\left(
    \begin{array}{cc}
      0 & 1\\
      a(t)\left[1+\frac{q(t)}{\nu^{2}(\xi)}\right] & 0 \\
           \end{array}
  \right).
$$
The initial conditions $\widehat{v}(0,\xi)=\widehat{v}_{0}(\xi)$, $\partial_{t}\widehat{v}(0,\xi)=\widehat{v}_{1}(\xi)$
are transformed into
$$
V(0,\xi)=\left(
    \begin{array}{c}
      i\nu(\xi) \widehat{v}_0(\xi)\\
      \widehat{v}_{1}(\xi)
     \end{array}
  \right).
$$

Note that the matrix $A$ has eigenvalues $\pm\sqrt{a(t)\left[1+\frac{q(t)}{\nu^{2}(\xi)}\right]}$ and its symmetriser is given by
\begin{equation}\label{EQ:Sdef}
S(t,\xi)=\left(
    \begin{array}{cc}
      a(t)\left[1+\frac{q(t)}{\nu^{2}(\xi)}\right] & 0\\
      0 & 1 \\
           \end{array}
  \right),
\end{equation}
i.e. we have
\[
SA-A^\ast S=0.
\]
It is immediate to prove that
\begin{equation}
\label{est_sym}
\min_{t\in[0,T]}(a(t)\left[1+\frac{q(t)}{\nu^{2}(\xi)}\right],1)|V|^2\le (SV,V)\le \max_{t\in[0,T]}(a(t)\left[1+\frac{q(t)}{\nu^{2}(\xi)}\right],1)|V|^2,
\end{equation}
where $(\cdot,\cdot)$ and $|\cdot|$ denote the inner product and the norm in $\C$, respectively.

Since $a(t)>a_0\geq0$, $q(t)\geq 0$, and $a,q\in C([0,T])$, it is clear that there exist constants $a_1>0$ and $a_2>0$ such that
$$
a_1=\min_{t\in[0,T]}a(t)\left[1+\frac{q(t)}{\nu^{2}(\xi)}\right]
\; \textrm{ and } \;
a_2=\max_{t\in[0,T]}{a(t)}\left[1+\frac{q(t)}{\nu^{2}(\xi)}\right].
$$
Hence \eqref{est_sym} implies that
\begin{equation}
\label{est_sym_1}
c_1|V|^2=\min(a_0,1)|V|^2\le (SV,V)\le \max(a_1,1)|V|^2=c_2|V|^2,
\end{equation}
with $c_1,c_2>0$.
We then define the energy
$$E(t,\xi):=(S(t,\xi)V(t,\xi),V(t,\xi)).$$
We get, from \eqref{est_sym_1}, that
\begin{align*}
\partial_t E(t,\xi)&=(\partial_t S(t,\xi)V(t,\xi),V(t,\xi))+(S(t,\xi)\partial_t V(t,\xi),V(t,\xi))\\
&\,\,\,\,\,\,\,\,\,\,\,\,\,\,\,\,\,\,\,\,\,\,\,\,\,\,\,\,\,\,\,\,\,\,\,\,\,\,\,\,\,\,\,\,\,\,\,\,\,\,\,\,\,\,\,\,\,\,\,\,\,\,\,\,
+(S(t,\xi)V(t,\xi),\partial_t V(t,\xi))\\
&=(\partial_t S(t,\xi)V(t,\xi),V(t,\xi))+i \nu(\xi) (S(t,\xi)A(t,\xi)V(t,\xi),V(t,\xi))\\
&\,\,\,\,\,\,\,\,\,\,\,\,\,\,\,\,\,\,\,\,\,\,\,\,\,\,\,\,\,\,\,\,\,\,\,\,\,\,\,\,\,\,\,\,\,\,\,\,\,\,\,\,\,\,\,\,\,\,\,\,\,\,\,\,
-i \nu(\xi) (S(t,\xi)V(t,\xi),A(t,\xi)V(t,\xi))\\
&=(\partial_t S(t,\xi)V(t,\xi),V(t,\xi))+i \nu(\xi) ((SA-A^\ast S)(t,\xi)V(t,\xi),V(t,\xi))\\
&=(\partial_t S(t,\xi)V(t,\xi),V(t,\xi))\\
&\le \Vert \partial_t S\Vert |V(t,\xi)|^2.
\end{align*}
Since $a(t)\left[1+\frac{q(t)}{\nu^{2}(\xi)}\right]$ is bounded on $[0, T]$ and for all $\xi$,
we obtain
\begin{equation}
\label{E_1}
\partial_t E(t,\xi)\le c' E(t,\xi),
\end{equation}
for some constant $c'>0$.
A part of the subsequent application of the Gronwall's lemma is standard (see e.g. \cite{M08}) but we give it for completeness and clarity.
By Gronwall's lemma applied to inequality
\eqref{E_1} we conclude that for all $T>0$ there exists $c>0$ such that
\[
E(t,\xi)\le c E(0,\xi).
\]
Hence, inequalities \eqref{est_sym_1} yield
\[
c_0|V(t,\xi)|^2\le E(t,\xi)\le c E(0,\xi)\le cc_1|V(0,\xi)|^2,
\]
for constants independent of $t\in[0,T]$ and $\xi$. This allows us to write the following statement:
there exists a constant $C_1>0$ such that
\begin{equation}
\label{case_1_est}
|V(t,\xi)|\le C_1 |V(0,\xi)|,
\end{equation}
for all $t\in[0,T]$ and $\xi$. Hence
$$
\nu^2(\xi) |\widehat{v}(t,\xi)|^2+|\partial_t \widehat{v}(t,\xi)|^2
\le C_1'( \nu^2(\xi) |\widehat{v}_0(\xi)|^2+|\widehat{v}_1(\xi)|^2).
$$
Recalling the notation
$\widehat{v}(t,\xi)=\widehat{u}(t,\xi)_{mk}$ and
$\nu^2(\xi)=(B+2B\xi_{2})$, this means
\begin{equation}
\label{case_1_est_mn} (B+2B\xi_{2})
|\widehat{u}(t,\xi)_{mk}|^2+|\partial_t \widehat{u}(t,\xi)_{mk}|^2
\le C_1'(B+2B\xi_{2})
|\widehat{u}_0(\xi)_{mk}|^2+|\widehat{u}_1(\xi)_{mk}|^2)
\end{equation}
for all $t\in[0,T]$, $\xi\in\N$ and $1\le m,k\le 2$, with the
constant $C_1'$ independent of $\xi$, $m,k$.
Now we recall that by Plancherel's equality, we have
$$
\|\partial_t u(t,\cdot)\|_{L^2}^2=\sum_{\xi\in\N}\|\partial_t \widehat{u}(t,\xi)\|_\HS^2=
\sum_{\xi\in\N} \sum_{m,k=1}^2 |\partial_t \widehat{u}(t,\xi)_{mk}|^2
$$
and
$$
\|\H^{1/2} u(t,\cdot)\|_{L^2}^2=\sum_{\xi\in\N}   \| \sigma_\H(\xi)^{1/2}
\widehat{u}(t,\xi)\|_\HS^2= \sum_{\xi\in\N} \sum_{m,k=1}^2
(B+2B\xi_{2}) |\widehat{u}(t,\xi)_{mk}|^2.
$$
Hence, the estimate \eqref{case_1_est_mn} implies that
\begin{equation}
\label{case_1_last}
\|\H^{1/2} u(t,\cdot)\|_{L^2}^2+\|\partial_t u(t,\cdot)\|_{L^2}^2\leq
C (\|\H^{1/2} u_0\|_{L^2}^2+\|u_1\|_{L^2}^2),
\end{equation}
where the constant $C>0$ does not depend on $t\in[0,T]$. More
generally, multiplying \eqref{case_1_est_mn} by powers of
$(B+2B\xi_{2})$, for any $s$, we get
\begin{multline}
\label{case_1_est_mn2} (B+2B\xi_{2})^{1+s}
|\widehat{u}(t,\xi)_{mk}|^2+(B+2B\xi_{2})^{s}  |\partial_t
\widehat{u}(t,\xi)_{mk}|^2 \\
\le C_1'(B+2B\xi_{2})^{1+s}
|\widehat{u}_0(\xi)_{mk}|^2+(B+2B\xi_{2})^{s}
|\widehat{u}_1(\xi)_{mk}|^2).
\end{multline}
Taking the sum over $\xi$, $m$ and $k$ as above, this yields
the estimate \eqref{case_1_last-est}.

\section{Proof of Theorem \ref{theo_vws}}
\label{SEC:pf2}

Again, in this section we deal with the Cauchy problem \eqref{CPa} and the proof for equation \eqref{CPb} can be done by minor modifications.

We now assume that the equation coefficients are distributions with compact support contained in $[0,T]$. Since the formulation of  \eqref{CPa}  in this case might be impossible in the distributional sense due to issues related to the product of distributions, we replace \eqref{CPa} with a regularised equation. In other words, we regularise $a, q$ by convolution with a mollifier in $C^\infty_0(\mathbb R)$ and get nets of smooth functions as coefficients. More precisely, let $\psi\in C^\infty_0(\mathbb R)$, $\psi\ge 0$ with $\int\psi=1$, and let $\omega(\eps)$ be a positive function converging to $0$ as $\eps\to 0$, with the rate of convergence to be specified later. Define
$$
\psi_{\omega(\eps)}(t):=\frac{1}{\omega(\eps)}\psi\left(\frac{t}{\omega(\eps)}\right),
$$
$$
a_{\eps}(t):=(a\ast \psi_{\omega(\eps)})(t), \; q_{\eps}(t):=(q\ast \psi_{\omega(\eps)})(t), \qquad t\in[0,T].
$$
Since $a$ is a positive distribution with compact support (hence a Radon measure) and $\psi\in C^\infty_0(\mathbb R)$, $\textrm{supp}\,\psi\subset\textsc{K}$, $\psi\ge 0$,
identifying the measure $a$ with its density, we can write
\begin{align*}
a_{\eps}(t)&=(a\ast \psi_{\omega(\eps)})(t)=\int\limits_{\mathbb R}a(t-\tau)\psi_{\omega(\eps)}(\tau)d\tau=\int\limits_{\mathbb R}a(t-\omega(\eps)\tau)\psi(\tau)d\tau \\
&=\int\limits_{\textsc{K}}a(t-\omega(\eps)\tau)\psi(\tau)d\tau\geq a_{0} \int\limits_{\textsc{K}}\psi(\tau)d\tau:=\tilde a_{0}>0,
\end{align*}
with a positive constant $\tilde a_{0}>0$ independent of $\eps$.

By the structure theorem for compactly supported distributions, we have that there exist $L_{1}, L_{2}\in\mathbb N$ and $c_{1}, c_{2}>0$ such that
\begin{equation}\label{EQ: a-q-C-moderate}
|\partial^{k}_{t}a_{\eps}(t)|\le c_{1}\,\omega(\eps)^{-L_{1}-k}, \,\,\, \,\,\, \,\,\, |\partial^{k}_{t}q_{\eps}(t)|\le c_{2}\,\omega(\eps)^{-L_{2}-k},
\end{equation}
for all $k\in\mathbb N_{0}$ and $t\in[0,T]$. We note that the numbers $L_{1}$ and $L_{2}$ may be related to the distributional orders of $a$ and $q$ but we will not be needing such a relation in our proof.

Hence, $a_{\eps}, q_{\eps}$ are $C^\infty$--moderate regularisations of the coefficients $a, q$.
Now, fix $\eps\in(0,1]$, and consider the regularised problem
\begin{equation}\label{PrTh2:EQ:1}
\left\{ \begin{split}
\partial_{t}^{2}u_{\eps}(t,x)+a_{\eps}(t)[\H+q_{\eps}(t)] u_{\eps}(t,x)&=0, \; (t,x)\in [0,T]\times \mathbb R^{2},\\
u_{\eps}(0,x)&=u_{0}(x), \; x\in \mathbb R^{2}, \\
\partial_{t}u_{\eps}(0,x)&=u_{1}(x), \; x\in \mathbb R^{2},
\end{split}
\right.
\end{equation}
with the Cauchy data satisfy
$(u_0,u_1)\in {H}^{1+s}_\H \times {H}^{s}_\H$ and $a_{\eps}\in C^\infty[0, T]$.
Then all discussions and calculations of Theorem \ref{theo_case_1} are valid. Thus by Theorem \ref{theo_case_1} the equation \eqref{PrTh2:EQ:1} has a unique solution in the space $C^{0}([0,T];{H}^{1+s}_\H)\cap C^{1}([0,T];{H}^{s}_\H)$. In fact, this unique solution is from $C^\infty([0,T];{H}^{s}_\H)$. This can be checked by taking in account that $a_{\eps}, q_{\eps}\in C^\infty([0,T])$ and by differentiating both sides of the equation \eqref{PrTh2:EQ:1} in $t$ inductively. Applying Theorem \ref{theo_case_1} to the equation \eqref{PrTh2:EQ:1},
using the inequality
$$
\Vert \partial_{t} S_{\eps}(t,\xi) \Vert \leq C(|\partial_{t} a_{\eps}(t)| |q_{\eps}(t)|+| a_{\eps}(t)| |\partial_{t} q_{\eps}(t)|) \leq C \omega(\eps)^{-L_{1}-L_{2}-1},
$$
with $S_{\eps}$ corresponding to \eqref{EQ:Sdef},
and Gronwall's lemma, we get the estimate
\begin{equation}\label{ES: exp-1}
\|u_{\eps}(t,\cdot)\|_{{H}^{1+s}_\H}^2+\| \partial_t u_{\eps}(t,\cdot)\|_{{H}^s_\H}^2\leq
C \exp(c\,\omega(\eps)^{-L_{1}-L_{2}-1}T) (\| u_0\|_{{H}^{1+s}_\H}^2+\|u_1\|_{{H}^{s}_\H}^2),
\end{equation}
where the coefficients $L_{1}$ and $L_{2}$ are from \eqref{EQ: a-q-C-moderate}.

Put $\omega(\eps)\sim\log^{-1}(\eps)$. Then the estimate \eqref{ES: exp-1} transforms to
$$
\|u_{\eps}(t,\cdot)\|_{{H}^{1+s}_\H}^2+\| \partial_t u_{\eps}(t,\cdot)\|_{{H}^s_\H}^2\leq
C \eps^{-L_{1}-L_{2}-1} (\| u_0\|_{{H}^{1+s}_\H}^2+\|u_1\|_{{H}^{s}_\H}^2),
$$
with possibly new constants $L_{1}, L_{2}$. To simplify the notation we continue denoting them by the same letters.

Now, let us show that there exist $N\in\mathbb N_{0}$, $c>0$ and, for all $k\in\mathbb N_{0}$ there exist $N_k>0$ and $c_k>0$ such that
$$
\|\partial_t^k u_\eps(t,\cdot)\|_{{H}^{s}_\H}\le c_k \eps^{-N-k},
$$
for all $t\in[0,T]$, and $\eps\in(0,1]$.

Applying \eqref{est_sym_1} and \eqref{E_1} to the problem with $a_{\eps}$ and $q_{\eps}$, and by
taking account the properties of $a_{\eps}$ and $q_{\eps}$, we get
\begin{align*}
(B+2B\xi_{2})
|\widehat{u_\eps}(t,\xi)_{mk}|^2 &+|\partial_t
\widehat{u_\eps}(t,\xi)_{mk}|^2 \\
& \le C \eps^{-L_{1}-L_{2}-1}((B+2B\xi_{2})
|\widehat{u}_0(\xi)_{mk}|^2+|\widehat{u}_1(\xi)_{mk}|^2)
\end{align*}
for all $t\in[0,T]$, $\xi\in\N$ and $1\le m,k\le 2$, with the
constant $C$ independent of $\xi$, $m,k$. Thus, we obtain
$$
\|\partial_t u_\eps(t,\cdot)\|_{{H}^{s}_\H} \le C \eps^{-L_{1}-L_{2}-1}, \,\,\, \| u_\eps(t,\cdot)\|_{{H}^{s+1}_\H}\le
C \eps^{-L_{1}-L_{2}}.
$$
Acting by the iterations of $\partial_{t}$ and by $\H$ on the
equality
$$
\partial_{t}^{2}u_{\eps}(t,x)=a_{\eps}(t)[\H+q_{\eps}(t)] u_{\eps}(t,x),
$$
and taking it in $L^{2}$--norms, we conclude that $u_{\eps}$ is
$C^\infty([0,T];{H}^{s}_\H)$-moderate.

This shows that the Cauchy problem \eqref{CPa} has a very weak solution.

\section{Consistency with the classical well-posedness}
\label{SEC:cons}

In this section we show that when the coefficients are
regular enough then the very weak solution coincides with the
classical one: this is the content of Theorem \ref{theo_consistency-2} which we will prove here.

Moreover, we show that the very weak solution provided by Theorem \ref{theo_vws}
is unique in an appropriate sense. For formulating the uniqueness statement it will be convenient to use the language of Colombeau algebras.

\begin{defi}
\label{def_negl_net} We say that $(u_\eps)_\eps$ is
\emph{$C^\infty$-negligible} if for all $K\Subset\mathbb R$, for
all $\alpha\in\mathbb N$ and for all $\ell\in\mathbb N$ there exists
a constant $c>0$ such that
$$
\sup_{t\in K}|\partial^\alpha u_\eps(t)|\le c\eps^{\ell},
$$
for all $\eps\in(0,1]$.
\end{defi}

We now introduce the Colombeau algebra as the quotient
$$
\mathcal G(\mathbb R)=\frac{C^\infty-\text{moderate\,
nets}}{C^\infty-\text{negligible\, nets}}.
$$
For the general analysis of $\mathcal G(\mathbb R)$ we refer to
e.g. Oberguggenberger \cite{Oberguggenberger:Bk-1992}.

\begin{thm}[Uniqueness]
\label{theo_consistency-1}
Let $a$ and $q$ be positive distributions
with compact support included in $[0,T]$, such that $a\ge a_{0}$ for some constant $a_{0}>0$.
Let $(u_0,u_1)\in {H}^{1+s}_\H \times {H}^{s}_\H$ for some $s\in\mathbb R$.
Then there exists an embedding of the coefficients $a$ and $q$ into $\mathcal G([0,T])$,
such that the Cauchy problem \eqref{CPa}, that is
\begin{equation*}
\left\{ \begin{split}
\partial_{t}^{2}u(t,x)+a(t)[\H+q(t)] u(t,x)&=0, \; (t,x)\in [0,T]\times \mathbb R^{2},\\
u(0,x)&=u_{0}(x), \; x\in \mathbb R^{2}, \\
\partial_{t}u(0,x)&=u_{1}(x), \; x\in \mathbb R^{2},
\end{split}
\right.
\end{equation*}
has a unique solution $u\in
\mathcal G([0,T]; H^s_{\H})$ for all $s\in\mathbb R$.

The same statement holds also for the Cauchy problem
\eqref{CPb}.
\end{thm}

Here $\mathcal G([0,T]; H^s_{\H})$ stands for the space of families which are
in $\mathcal G([0,T])$ with respect to $t$ and in $H^s_{\H}$ with respect to $x$.

\begin{proof}
Let us show that by embedding coefficients in the
corresponding Colombeau algebras the Cauchy problem has a unique
solution $u\in\mathcal G([0,T]; H^s_{\H})$. Assume now that the Cauchy problem has another
solution $v\in\mathcal G([0,T]; H^s_{\H})$. At the level of
representatives this means
\begin{equation*}
\left\{ \begin{split}
\partial_{t}^{2}(u_\eps-v_\eps)(t,x)+a_\eps(t)[\H+q_{\eps}(t)] (u_\eps-v_\eps)(t,x)&=f_{\eps}(t,x), \\
(u_\eps-v_\eps)(0,x)&=0,  \\
(\partial_{t}u_\eps-\partial_{t}v_\eps)(0,x)&=0,
\end{split}
\right.
\end{equation*}
with
$$
f_{\eps}(t,x)=(a_\eps(t)-\widetilde{a}_\eps(t))\H v_\eps(t,x)+(a_\eps(t)q_{\eps}(t)-\widetilde{a}_\eps(t)\widetilde{q}_{\eps}(t)) v_\eps(t,x),
$$
where $(\widetilde{a}_\eps)_{\eps}$ and $(\widetilde{q}_\eps)_{\eps}$ are approximations corresponding to $v_\eps$. It is obvious, that $f_{\eps}$ is $C^{\infty}([0,T]; H^s_{\H})$--negligible.
The corresponding first order system is
$$
\partial_t\left(
                             \begin{array}{c}
                               W_{1,\eps} \\
                                W_{2,\eps} \\
                             \end{array}
                           \right)
= \left(
    \begin{array}{cc}
      0 & i\H^{1/2}\\
      i a_{\eps}(t)[\H^{1/2}+q_{\eps}(t)\H^{-1/2}]& 0 \\
           \end{array}
  \right)
  \left(\begin{array}{c}
                               W_{1,\eps} \\

                               W_{2,\eps} \\
                             \end{array}
                           \right)+\left(\begin{array}{c}
                               0 \\

                               f_\eps \\
                             \end{array}
                           \right),
$$
where $W_{1,\eps}$ and $W_{2,\eps}$ are obtained via the
transformation
$$
W_{1,\eps}=\H^{1/2}(u_\eps-v_\eps),\,\,\,
W_{2,\eps}=\partial_t(u_\eps-v_\eps).
$$
This system will be studied after $\H$--Fourier transform, as a system
of the type
\begin{equation*}
\partial_t V_{\eps}(t,\xi)=i \nu(\xi) A_{\eps}(t, \xi)V_{\eps}(t,\xi)+F_{\eps}(t,\xi),
\end{equation*}
with
$$
F_\eps=\left(\begin{array}{c} 0 \\
    \mathcal{F}_{\H}f_\eps \\
             \end{array}
       \right),
$$
and
$$
A_{\eps}(t, \xi)=\left(
    \begin{array}{cc}
      0 & 1\\
      a_{\eps}(t)\left[1+\frac{1}{\nu^2(\xi)}q_{\eps}(t)\right] & 0 \\
           \end{array}
  \right),
$$
with Cauchy data
$$
V_{\eps}(0, \xi)=\left(
    \begin{array}{cc}
      0\\
      0 \\
           \end{array}
  \right).
$$
For the symmetriser
$$
S_{\eps}(t, \xi)=\left(
    \begin{array}{cc}
      a_{\eps}(t)\left[1+\frac{1}{\nu^2(\xi)}q_{\eps}(t)\right] & 0\\
      0 & 1 \\
           \end{array}
  \right)
$$
 define the energy
$$E_{\eps}(t,\xi):=(S_{\eps}(t, \xi)V_{\eps}(t,\xi),V_{\eps}(t,\xi)).$$
We get
\begin{align*}
\partial_t E_{\eps}(t,\xi)&=(\partial_t S_{\eps}(t, \xi)V_{\eps}(t,\xi),V_{\eps}(t,\xi))+(S_{\eps}(t, \xi)\partial_t V_{\eps}(t,\xi),V_{\eps}(t,\xi))\\
&\,\,\,\,\,\,\,\,\,\,\,\,\,\,\,\,\,\,\,\,\,\,\,\,\,\,\,\,\,\,\,\,\,\,\,\,\,\,\,\,\,\,\,\,\,\,\,\,\,\,\,\,\,\,\,\,\,\,\,\,\,\,\,\,
+(S_{\eps}(t, \xi)V_{\eps}(t,\xi),\partial_t V_{\eps}(t,\xi))\\
&=(\partial_t S_{\eps}(t, \xi)V_{\eps}(t,\xi),V_{\eps}(t,\xi))\\
&+i \nu(\xi)
(S_{\eps}(t, \xi)A_{\eps}(t, \xi)V_{\eps}(t,\xi),V_{\eps}(t,\xi))-i \nu(\xi) (S_{\eps}(t, \xi)V_{\eps}(t,\xi), A_{\eps}(t, \xi)V_{\eps}(t,\xi))\\
&+(S_{\eps}(t, \xi)F_{\eps}(t,\xi),V_{\eps}(t,\xi))+(S_{\eps}(t, \xi)V_{\eps}(t,\xi), F_{\eps}(t,\xi))\\
&=(\partial_t S_{\eps}(t, \xi)V_{\eps}(t,\xi),V_{\eps}(t,\xi))+i \nu(\xi) ((S_{\eps}A_{\eps}-A^{\ast}_{\eps} S_{\eps})(t, \xi)V_{\eps}(t,\xi),V_{\eps}(t,\xi))\\
&+(S_{\eps}(t, \xi)F_{\eps}(t,\xi),V_{\eps}(t,\xi))+(V_{\eps}(t,\xi), S_{\eps}(t, \xi) F_{\eps}(t,\xi))\\
&=(\partial_t S_{\eps}(t, \xi)V_{\eps}(t,\xi),V_{\eps}(t,\xi))+2\textrm{Re}(S_{\eps}(t, \xi)F_{\eps}(t,\xi),V_{\eps}(t,\xi))\\
&\le \Vert \partial_t S_{\eps}\Vert |V_{\eps}(t,\xi)|^2+2\textrm{Re}(S_{\eps}(t, \xi)F_{\eps}(t,\xi),V_{\eps}(t,\xi))\\
&\le \Vert \partial_t S_{\eps}\Vert |V_{\eps}(t,\xi)|^2+2\Vert
S_{\eps}\Vert |F_{\eps}(t,\xi)| |V_{\eps}(t,\xi)|.
\end{align*}
Assuming for the moment that $|V_{\eps}(t,\xi)|>1$, we get the
energy estimate
\begin{align*}
\partial_t E_{\eps}(t,\xi)&\le \Vert \partial_t S_{\eps}\Vert |V_{\eps}(t,\xi)|^2+2\Vert
S_{\eps}\Vert |F_{\eps}(t,\xi)| |V_{\eps}(t,\xi)|\\
&\le (\Vert \partial_t S_{\eps}\Vert +2\Vert
S_{\eps}\Vert |F_{\eps}(t,\xi)|)|V_{\eps}(t,\xi)|^2\\
&\le \left(|\partial_t a_{\eps}(t)||q_{\eps}(t)|+2|a_{\eps}(t)||\partial_t q_{\eps}(t)|+|a_{\eps}(t)||q_{\eps}(t)| |F_{\eps}(t,\xi)|\right)|V_{\eps}(t,\xi)|^2\\
&\le c\, \omega(\eps)^{-L_{1}-L_{2}-1} E_{\eps}(t,\xi),
\end{align*}
i.e. we obtain
\begin{equation}
\label{EQ:E_1}
\partial_t E_{\eps}(t,\xi)\le c\, \omega(\eps)^{-L_{1}-L_{2}-1} E_{\eps}(t,\xi),
\end{equation}
for some constant $c>0$. By Gronwall's lemma applied to
inequality \eqref{EQ:E_1} we conclude that for all $T>0$
$$
E_{\eps}(t,\xi)\le \exp(c\, \omega(\eps)^{-L_{1}-L_{2}-1} \, T) E_{\eps}(0,\xi).
$$
Hence, inequalities \eqref{est_sym_1} yield
\begin{align*}
c_0|V_{\eps}(t,\xi)|^2\le E_{\eps}(t,\xi)&\le \\
&\le \exp(c\, \omega(\eps)^{-L_{1}-L_{2}-1}\, T)  E_{\eps}(0,\xi) \\
&\le \exp(c_1\, \omega(\eps)^{-L_{1}-L_{2}-1}\, T) |V_{\eps}(0,\xi)|^2,
\end{align*}
for the constant $c_{1}$ independent of $t\in[0,T]$ and $\xi$.

By putting $\omega(\eps)\sim\log^{-1}(\eps)$, we get
\begin{align*}
|V_{\eps}(t,\xi)|^2\le c\, \eps^{-L_{1}-L_{2}-1} |V_{\eps}(0,\xi)|^2
\end{align*}
for some constant $c$ and some (new) $L_{1},L_{2}$.
Since $|V_{\eps}(0,\xi)|=0$, we have
$$
|V_{\eps}(t,\xi)|\equiv0,
$$
for all $\xi$ and for $t\in[0,T]$.

Now consider the case when $|V_{\eps}(t,\xi)|<1$. Assume that $|V_{\eps}(t,\xi)|\geq c \, \omega(\varepsilon)^{\alpha}$ for some constant $c$ and $\alpha>0$. It means
$$
\frac{1}{|V_{\eps}(t,\xi)|}\leq C \, \omega(\varepsilon)^{-\alpha}.
$$
Then the estimate for the energy becomes
\begin{align*}
\partial_t E_{\eps}(t,\xi)&\le C\, \omega(\varepsilon)^{-L} E_{\eps}(t,\xi),
\end{align*}
where $L=L_{1}+L_{2}+\textrm{max}\{1, \alpha\}$, and by Gronwall's lemma
$$
|V_{\eps}(t,\xi)|^2\le \exp(C'\, \omega(\varepsilon)^{-L} \, T) |V_{\eps}(0,\xi)|^2.
$$
And again, by putting $\omega(\eps)\sim\log^{-1}(\eps)$, we get
$$
|V_{\eps}(t,\xi)|^2\le c'\, \varepsilon^{-L} |V_{\eps}(0,\xi)|^2
$$
for some $c'$ and some (new) $L$.
Since $|V_{\eps}(0,\xi)|=0$, we have
$$
|V_{\eps}(t,\xi)|\equiv0,
$$
for all $t\in[0, T]$ and $\xi$.

The last case is when $|V_{\eps}(t,\xi)|\leq c \, \omega(\varepsilon)^{\alpha}$ for some constant $c$ and $\alpha>0$. Indeed, it completes the proof of Theorem \ref{theo_consistency-1}.
\end{proof}

\begin{proof}[Proof of Theorem \ref{theo_consistency-2}]
 We now want to compare the classical solution $\widetilde{u}$ given by Theorem \ref{theo_case_1} with the very weak solution $u$ provided by  Theorem \ref{theo_consistency-2}. By the
definition of the classical solution we know that
\begin{equation}\label{Consistency:EQ:2}
\left\{ \begin{split}
\partial_{t}^{2}\widetilde{u}(t,x)+a(t)[\H+q(t)] \widetilde{u}(t,x)&=0, \\
\widetilde{u}(0,x)&=u_{0}(x), \\
\partial_{t}\widetilde{u}(0,x)&=u_{1}(x).
\end{split}
\right.
\end{equation}
By the definition of the very weak solution $u$, there exists a representative $(u_\eps)_\eps$ of $u$ such
that
\begin{equation}\label{Consistency:EQ:3}
\left\{ \begin{split}
\partial_{t}^{2}u_{\eps}(t,x)+a_\eps(t)[\H+q_\eps(t)] u_{\eps}(t,x)&=0, \\
u_{\eps}(0,x)&=u_{0}(x),  \\
\partial_{t}u_{\eps}(0,x)&=u_{1}(x),
\end{split}
\right.
\end{equation}
for a suitable embedding of the coefficients $a$ and $q$. Noting that for $a, q\in L_{1}^{\infty}([0,T])$ the
nets $(a_{\eps}-a)_\eps$ and $(q_{\eps}-q)_\eps$ are converging to $0$ in
$C([0,T]\times\mathbb R^2)$, we can rewrite
\eqref{Consistency:EQ:2} as
\begin{equation}\label{Consistency:EQ:4}
\left\{ \begin{split}
\partial_{t}^{2}\widetilde{u}(t,x)+a_\eps(t)[\H+q_\eps(t)] \widetilde{u}(t,x)&=n_{\eps}(t,x), \\
\widetilde{u}(0,x)&=u_{0}(x),  \\
\partial_{t}\widetilde{u}(0,x)&=u_{1}(x),
\end{split}
\right.
\end{equation}
where $n_{\eps}(t,x)=[(a_\eps(t)-a(t))\H+(a_\eps(t)q_{\eps}(t)-a(t)q(t)) n_\eps(t,x),$
and $n_\eps\in C([0,T]; H^s_{\H})$ and converges to $0$ in this space as $\eps\to0$. From \eqref{Consistency:EQ:3} and \eqref{Consistency:EQ:4} we get that $\widetilde{u}-u_\eps$ solves the Cauchy problem
\begin{equation*}
\left\{ \begin{split}
\partial_{t}^{2}(\widetilde{u}-u_\eps)(t,x)+a_\eps(t)[\H+q_\eps(t)](\widetilde{u}-u_\eps)(t,x)&=n_{\eps}(t,x), \\
(\widetilde{u}-u_\eps)(0,x)&=0,  \\
(\partial_{t}\widetilde{u}-\partial_{t}u_\eps)(0,x)&=0.
\end{split}
\right.
\end{equation*}
As in the first part of the proof we arrive, after
reduction to a system and by application of the Fourier transform
to estimate $|(\widetilde{V}-V_\eps)(t,\xi)|$ in
terms of $(\widetilde{V}-V_\eps)(0,\xi)$ and the right-hand side
$n_\eps(t,x)$, to the energy estimate
\begin{align*}
\partial_t E_{\eps}(t,\xi)\le &\left(|\partial_t a_{\eps}(t)||q_{\eps}(t)|+|a_{\eps}(t)||\partial_t q_{\eps}(t)|\right) |(\widetilde{V}-V_\eps)(t,\xi)|^2\\
&+2|a_{\eps}(t)| |q_{\eps}(t)| |n_{\eps}(t,\xi)| |(\widetilde{V}-V_\eps)(t,\xi)|.
\end{align*}
Since the coefficients are regular
enough, we simply get
\begin{align*}
\partial_t E_{\eps}(t,\xi)&\le c_{1}\, |(\widetilde{V}-V_\eps)(t,\xi)|^2+c_{2}\, |n_{\eps}(t,\xi)| |(\widetilde{V}-V_\eps)(t,\xi)|.
\end{align*}
Since  $(\widetilde{V}-V_\eps)(0,\xi)=0$ and $n_\eps\to 0$ in
$C([0,T];H^s_{\H})$ and continuing to discussing as in Theorem \ref{theo_consistency-1} we conclude that
$|(\widetilde{V}-V_\eps)(t,\xi)|\leq c\, \omega(\varepsilon)^{\alpha}$ for some constant $c$ and $\alpha>0$. Then we have $u_\eps\to \widetilde{u}$
in $C([0,T],{H}^{1+s}_\H) \cap
C^1([0,T],{H}^{s}_\H)$.
Moreover, since any other
representative of $u$ will differ from $(u_\eps)_\eps$ by a
$C^\infty([0,T];H^s_{\H})$-negligible net, the limit is the
same for any representative of $u$.
\end{proof}

\section{Inhomogeneous equation case}
\label{SEC:Inhom}

In this section we are going to give brief ideas for how to deal with the inhomogeneous wave equation
\begin{equation}\label{CPa-InH}
\left\{ \begin{split}
\partial_{t}^{2}u(t,x)+a(t)[\H+q(t)] u(t,x)&=f(t,x), \; (t,x)\in [0,T]\times \mathbb R^{2},\\
u(0,x)&=u_{0}(x), \; x\in \mathbb R^{2}, \\
\partial_{t}u(0,x)&=u_{1}(x), \; x\in \mathbb R^{2},
\end{split}
\right.
\end{equation}
where $a=a(t)\geq 0$ is a distributional propagation speed function, $q=q(t)$ is the distributional electromagnetic scalar potential, $f=f(t,x)$ is the distributional source term, and $\H$ is the Landau Hamiltonian.

\begin{thm}
\label{theo_case_1-InH}
Given $f\in C([0,T]; {H}^{s}_\H)$. Assume that $a, q\in L_{1}^{\infty}([0,T])$ are such that $a(t)\ge a_0>0$ and $q(t)\geq 0$.
For any $s\in\Rr$, if the Cauchy data satisfy
$(u_0,u_1)\in {H}^{1+s}_\H \times {H}^{s}_\H$,
then the Cauchy problem \eqref{CPa-InH} has a unique solution
$u\in C([0,T]; {H}^{1+s}_\H) \cap
C^1([0,T]; {H}^{s}_\H)$ which satisfies the estimate
\begin{equation}
\label{case_1_last-est2}
\|u(t,\cdot)\|_{{H}^{1+s}_\H}^2+\| \partial_t u(t,\cdot)\|_{{H}^s_\H}^2\leq
C (\| u_0\|_{{H}^{1+s}_\H}^2+\|u_1\|_{{H}^{s}_\H}^2+\sup_{0\leq t\leq T}\|f(t, \cdot)\|_{{H}^{s}_\H}^2).
\end{equation}
\end{thm}

Keeping the notations the same as in the proof of Theorem \ref{theo_case_1}, we can write equation \eqref{CPa-InH} as the first order system
\begin{equation}\label{EQ:system-InH}
\partial_t V(t,\xi)=i\nu(\xi) A(t,\xi)V(t,\xi)+F(t, \xi),
\end{equation}
where
$$
F(t, \xi)=\left(
    \begin{array}{c}
      0\\
      \widehat{f}(t, \xi)
     \end{array}
  \right),
$$
and for the energy $E(t,\xi):=(SV,V)$ we get
\begin{align*}
\partial_t E(t,\xi)&=(\partial_t SV,V)+(S\partial_t V,V)
+(SV,\partial_t V) \\
&=(\partial_t SV,V)-2\mathrm{Im}(SF,V) \\
&\le (\Vert \partial_t S\Vert+1) |V|^2+ \Vert SF \Vert^{2} \\
&\le \mathrm{max}(\Vert \partial_t S\Vert+1, \Vert S \Vert^{2}) (|V|^2+|F|^{2}) \\
&\le C_{1} E(t,\xi)+ C_{2} |F|^{2}
\end{align*}
with some constants $C_{1}$ and $C_{2}$. An application of Cronwall's lemma combined with the estimates \eqref{est_sym_1} implies
$$
|V|^{2}\leq c_{1}^{-1} E(t,\xi)\leq C_{1} |V_{0}|^{2}+ C_{2} \sup_{0\leq t\leq T}|F(t, \xi)|^{2},
$$
which is valid for all $t\in[0, T]$ with `new' constants $C_{1}$ and $C_{2}$ depending on $T$. By continuing our discussion as in the proof of Theorem \ref{theo_case_1}, we prove Theorem \ref{theo_case_1-InH}.

Let us formulate definition of the very weak solution for the inhomogeneous wave equation \eqref{CPa-InH}.
\begin{defi}
\label{def_vws-InH}
Let $s\in\mathbb R$, $f\in C([0,T]; {H}^{s}_\H)$ and $u_{0},u_{1}\in {H}^{s}_\H$. The net $(u_\eps)_\eps\in C^\infty([0,T];{H}^{s}_\H)$ is {\em a very weak solution of order $s$} of the
Cauchy problem \eqref{CPa-InH} if there exist
\begin{itemize}
\item[]
$C^\infty$-moderate regularisations $a_{\eps}$ and $q_{\eps}$ of the coefficients $a$ and $q$,
\item[]
$C^\infty([0,T];{H}^{s}_\H)$-moderate regularisation $f_{\eps}$ of the source term $f$,
\end{itemize}
such that $(u_\eps)_\eps$ solves the regularised problem
\begin{equation*}\label{CPbb-InH}
\left\{ \begin{split}
\partial_{t}^{2}u_{\eps}(t,x)+a_{\eps}(t)[\H+q_{\eps}(t)] u_{\eps}(t,x)&=f_{\eps}(t, x), \; (t,x)\in [0,T]\times \mathbb R^{2},\\
u_{\eps}(0,x)&=u_{0}(x), \; x\in \mathbb R^{2}, \\
\partial_{t}u_{\eps}(0,x)&=u_{1}(x), \; x\in \mathbb R^{2},
\end{split}
\right.
\end{equation*}
for all $\eps\in(0,1]$, and is $C^\infty([0,T];{H}^{s}_\H)$-moderate.
\end{defi}

Without significant changes in the proofs of Theorems \ref{theo_vws}, \ref{theo_consistency-2} and \ref{theo_consistency-1}, we conclude the following modified results for the Cauchy problem \eqref{CPa-InH} for the inhomogeneous wave equation.

\begin{thm}[Existence]
\label{theo_vws-InH}
Let the coefficients $a$ and $q$ of the Cauchy problem \eqref{CPa-InH} be positive distributions
with compact support included in $[0,T]$, such that $a\ge a_{0}$ for some constant $a_{0}>0$, and  let the source term $f(\cdot, x)$ be a distribution with compact support included in $[0,T]$.
Let $s\in\mathbb R$ and let the Cauchy data $(u_0, u_1)$ be in ${H}^{s+1}_\H\times {H}^{s}_\H$ and the source term $f(t, \cdot)$ be in ${H}^{s}_\H$.
Then the Cauchy problem \eqref{CPa-InH} has a very weak solution of order $s$.
\end{thm}

Now let us formulate the theorem saying that very weak solutions recapture the classical solutions in the case the latter exist. This happens, for example, under conditions of Theorem \ref{theo_case_1-InH}. So, we can compare the solution given by Theorem \ref{theo_case_1-InH} with the very weak solution in Theorem \ref{theo_vws-InH} under assumptions when Theorem \ref{theo_case_1-InH} holds.

\begin{thm}[Consistency]
\label{theo_consistency-2-InH}
Assume that $a, q\in L_{1}^{\infty}([0,T])$ are such that $a(t)\ge a_0>0$ and $q(t)\geq 0$, and $f\in C([0,T],{H}^{s}_\H)$.
Let $s\in\mathbb R$, and consider the Cauchy problem
\begin{equation}\label{Consistency:EQ:1-2-InH}
\left\{ \begin{split}
\partial_{t}^{2}u(t,x)+a(t)[\H+q(t)] u(t,x)&=f(t, x), \; (t,x)\in [0,T]\times \mathbb R^{2},\\
u(0,x)&=u_{0}(x), \; x\in \mathbb R^{2}, \\
\partial_{t}u(0,x)&=u_{1}(x), \; x\in \mathbb R^{2},
\end{split}
\right.
\end{equation}
with $(u_0,u_1)\in {H}^{1+s}_\H \times {H}^{s}_\H$. Let $u$ be a very weak solution of
\eqref{Consistency:EQ:1-2-InH}. Then for any regularising families $a_{\eps}, q_{\eps}, f_{\eps}$ in Definition \ref{def_vws-InH}, any representative $(u_\eps)_\eps$ of $u$ converges in $C([0,T],{H}^{1+s}_\H) \cap
C^1([0,T],{H}^{s}_\H)$ as $\eps\rightarrow0$
to the unique classical solution in $C([0,T],{H}^{1+s}_\H) \cap
C^1([0,T],{H}^{s}_\H)$ of the Cauchy problem \eqref{Consistency:EQ:1-2-InH}
given by Theorem \ref{theo_case_1-InH}.
\end{thm}

\begin{thm}[Uniqueness]
\label{theo_consistency-1-InH}
Let $a$ and $q$ be positive distributions
with compact support included in $[0,T]$, such that $a\ge a_{0}$ for some constant $a_{0}>0$, and  the source term $f(\cdot, x)$ be a distribution with compact support included in $[0,T]$.
Let $(u_0,u_1)\in {H}^{1+s}_\H \times {H}^{s}_\H$ and the source term $f(t, \cdot)$ be in ${H}^{s}_\H$ for some $s\in\mathbb R$.
Then there exists an embedding of the coefficients $a$ and $q$ into $\mathcal G([0,T])$ and of $f$ into $\mathcal G([0,T]; H^s_{\H})$,
such that the Cauchy problem \eqref{CPa-InH}, that is
\begin{equation*}
\left\{ \begin{split}
\partial_{t}^{2}u(t,x)+a(t)[\H+q(t)] u(t,x)&=f(t, x), \; (t,x)\in [0,T]\times \mathbb R^{2},\\
u(0,x)&=u_{0}(x), \; x\in \mathbb R^{2}, \\
\partial_{t}u(0,x)&=u_{1}(x), \; x\in \mathbb R^{2},
\end{split}
\right.
\end{equation*}
has a unique solution $u\in
\mathcal G([0,T]; H^s_{\H})$ for all $s\in\mathbb R$.
\end{thm}

\section{Landau Hamiltonian in $\mathbb R^{2d}$}
\label{SEC:n}

Here we indicate a few changes in the multidimensional case compared to that in 2D.
Let $x=(x_{1},\ldots,x_{2d})\in\mathbb R^{2d}$ and setting all physical constants to be equal to $1$, in analogy to the case of $d=1$ in \eqref{eq:LandauHamiltonian}, let

\begin{equation} \label{eq:Hamiltonian2d}
\H:=\frac12 \p{i \nabla-\A}^{2},
\end{equation}
where
$$
\A=\p{- B_{1}x_{2}, B_{1}x_{1}, -B_{2}x_{4}, B_{2}x_{3},\ldots,-B_{d}x_{2d},B_{d}x_{2d-1}},
$$
corresponding to the magnetic fields of constant strengths $2B_{l}>0$, $l=1,\ldots,d$.
The essentially self-adjoint operator $\H$ on $C_{0}^{\infty}(\mathbb R^{2d})$ in the Hilbert space
$L^{2}(\mathbb R^{2d})=\otimes_{1}^{d}L^{2}(\mathbb R^{2})$ decomposes as
$$
\H=\H_{1}\otimes I^{\otimes (d-1)}+I\otimes \H_{2}\otimes I^{\otimes (d-2)}+\cdots+
 I^{\otimes (d-1)}\otimes \H_{d},
$$
with self-adjoint 2D operators $\H_{l}$ on $L^{2}(\mathbb R^{2})$ as in \eqref{eq:LandauHamiltonian}.
Let $k=(k_{1},\ldots, k_{d})\in\mathbb N_{0}^{d}$ be a multi-index.
Then in analogy to \eqref{eq:HamiltonianEigenvalues},
the spectrum of $\H$ consists of the infinitely degenerate eigenvalues
\begin{equation} \label{eq:HamiltonianEigenvalues2d}
\lambda_{k}=\sum_{l=1}^{d} B_{l}(2k_{l}+1),
\end{equation}
with eigenfunctions corresponding to \eqref{eq:HamiltonianBasis}.
In particular, in the isotropic case when $B_{l}=B>0$ for all $l$, for two multi-indices $k,k'\in \mathbb N_{0}^{d}$, if $|k|=|k'|$ then
$\lambda_{k}=\lambda_{k'}$ so that the spectrum of $\H$ consists of eigenvalues of the form $\lambda_{m}=B(2m+1)$ with $m\in\mathbb N_{0}.$
We refer e.g. to \cite{P09} and references therein for more details on the spectral analysis of this case.

We can now consider the Cauchy problem
\begin{equation}\label{CPa2d}
\left\{ \begin{split}
\partial_{t}^{2}u(t,x)+a(t)[\H+q(t)] u(t,x)&=0, \; (t,x)\in [0,T]\times \mathbb R^{2d},\\
u(0,x)&=u_{0}(x), \; x\in \mathbb R^{2d}, \\
\partial_{t}u(0,x)&=u_{1}(x), \; x\in \mathbb R^{2d},
\end{split}
\right.
\end{equation}
as a generalisation of \eqref{CPa}, with an analogous generalisation of \eqref{CPb}:
\begin{equation}\label{CPb2d}
\left\{ \begin{split}
\partial_{t}^{2}u(t,x)+a(t)\H u(t,x)+q(t) u(t,x)&=0, \; (t,x)\in [0,T]\times \mathbb R^{2d},\\
u(0,x)&=u_{0}(x), \; x\in \mathbb R^{2d}, \\
\partial_{t}u(0,x)&=u_{1}(x), \; x\in \mathbb R^{2d}.
\end{split}
\right.
\end{equation}
Let us indicate briefly the changes in the corresponding Fourier analysis generated by $\H$.
Thus, for any $s\in\Rr$, we set
$$
H^s_\H:=\left\{ f\in\Dcal'_{\H}(\mathbb R^{2d}): \H^{s/2}f\in
L^2(\mathbb R^{2d})\right\},
$$
with the norm $\|f\|_{H^s_\H}:=\|\H^{s/2}f\|_{L^{2}}$ given by
\begin{equation}\label{EQ:Sob12d}
\|f\|_{H^s_\H}= \p{\sum_{l=1}^{d} \sum_{\xi\in\N} (B_{l}+2B_{l}\xi_{2})^{s}
\sum_{j=1}^{2} \left| \int_{\mathbb
R^{2}}f(x)\overline{e^{j}_{l,\xi}(x)}dx\right|^2  }^{1/2},
\end{equation}
where for each $l$ the function $e_{l,\xi}^{j}$ is as in \eqref{eq:HamiltonianBasis}-\eqref{EQ:eks} with $B_{l}$ instead of $B$.

Consequently, all the statements of Theorem \ref{theo_case_1}, Theorem \ref{theo_consistency-2}
and Theorem \ref{theo_consistency-1} continue to hold for the Cauchy problems \eqref{CPa2d} and \eqref{CPb2d}. Namely, we have

\begin{thm}[Classical solutions]
\label{theo_case_12d}
Assume that $a, q\in L_{1}^{\infty}([0,T])$ are such that $a(t)\ge a_0>0$ and $q(t)\geq 0$.
For any $s\in\Rr$, if the Cauchy data satisfy
$(u_0,u_1)\in {H}^{1+s}_\H \times {H}^{s}_\H$,
then the Cauchy problems \eqref{CPa2d} and \eqref{CPb2d} have unique solutions
$u\in C([0,T],{H}^{1+s}_\H) \cap
C^1([0,T],{H}^{s}_\H)$ satisfying the estimate
\begin{equation}
\label{case_1_last-est2d}
\|u(t,\cdot)\|_{{H}^{1+s}_\H}^2+\| \partial_t u(t,\cdot)\|_{{H}^s_\H}^2\leq
C (\| u_0\|_{{H}^{1+s}_\H}^2+\|u_1\|_{{H}^{s}_\H}^2).
\end{equation}
\end{thm}
 We also have the corresponding very weak solutions result.

 \begin{thm}[Very weak solutions]
\label{theo_vws2d}
Let the coefficients $a$ and $q$ of the Cauchy problem \eqref{CPa2d} be positive distributions
with compact support included in $[0,T]$, such that $a\ge a_{0}$ for some constant $a_{0}>0$.
Let $s\in\mathbb R$ and let the Cauchy data $u_0, u_1$ be in ${H}^{s}_\H$. Then we have the following statements:

\begin{itemize}
\item (Existence) The Cauchy problem \eqref{CPa2d} has a very weak solution of order $s$.

\item (Uniqueness)
There exists an embedding of the coefficients $a$ and $q$ into $\mathcal G([0,T])$
such that the Cauchy problem \eqref{CPa2d} has a unique solution $u\in
\mathcal G([0,T]; H^s_{\H})$.

\item (Consistency)
Let $u$ be a very weak solution of \eqref{CPa2d}.
If $a, q\in L_{1}^{\infty}([0,T])$ are such that $a(t)\ge a_0>0$ and $q(t)\geq 0$, then
for any regularising families $a_{\eps}, q_{\eps}$, any representative $(u_\eps)_\eps$ of $u$ converges in $C([0,T],{H}^{1+s}_\H) \cap
C^1([0,T],{H}^{s}_\H)$ as $\eps\rightarrow0$
to the unique classical solution in $C([0,T],{H}^{1+s}_\H) \cap
C^1([0,T],{H}^{s}_\H)$ of the Cauchy problem \eqref{CPa2d}
given by Theorem \ref{theo_case_12d}.
\end{itemize}

The same result is true also for the Cauchy problem \eqref{CPb2d}.
\end{thm}

These theorems follow by an easy adaptation of the corresponding 2D proofs so we omit them.

\end{document}